\documentclass[12pt]{amsart}
\usepackage{amssymb}
\usepackage{amsmath}
\usepackage[total={6.25truein,9truein},centering]{geometry} 
\usepackage[bookmarks=true,colorlinks=true]{hyperref}
\usepackage{cleveref}

\crefformat{section}{\S#2#1#3} 
\crefformat{subsection}{\S#2#1#3}
\crefformat{subsubsection}{\S#2#1#3}

\numberwithin{equation}{section}

\theoremstyle{plain}
\newtheorem{thm}{Theorem}[section]
\newtheorem{lemma}[thm]{Lemma}
\newtheorem{prop}[thm]{Proposition}
\newtheorem{cor}[thm]{Corollary}

\theoremstyle{definition}
\newtheorem{defn}[thm]{Definition}
\newtheorem*{ack}{Acknowledgements}

\theoremstyle{remark}
\newtheorem{rmk}[thm]{Remark}
\newtheorem*{exmp}{Example}

\allowdisplaybreaks

\newcommand{\CC}{\mathbb{C}}
\newcommand{\FF}{\mathbb{F}}

\newcommand{\NN}{\mathbb{N}}

\newcommand{\RR}{\mathbb{R}}

\newcommand{\ZZ}{\mathbb{Z}}

\newcommand{\bs}{\mathbf{s}}

\newcommand{\cP}{\mathcal{P}}

\newcommand{\fJ}{\mathfrak{J}}

\newcommand\numberthis{\addtocounter{equation}{1}\tag{\theequation}}

\DeclareMathOperator{\Mat}{Mat}

\DeclareMathOperator{\re}{Re}

\DeclareMathOperator{\wt}{wt}

\newcommand{\ebar}{\bar{e}}

\newcommand{\bld}[1]{\mathbf{#1}}
\newcommand{\ceil}[1]{\lceil #1 \rceil}
\newcommand{\floor}[1]{\lfloor #1 \rfloor}
\newcommand{\inner}[2]{\langle #1, #2 \rangle}

\newcommand{\smod}[1]{{\ \mathrm{mod}\ #1}}

\begin{document}

\title{Vanishing of Multizeta Values over $\FF_q[t]$ at Negative Integers}

\author{Shuhui Shi}
\address{Department of Mathematics, Texas A{\&}M University, College Station,
TX 77843, USA}
\email{shuhui@math.tamu.edu}




\begin{abstract}
Let $\FF_q$ be the finite field of $q$ elements. In this paper, we study the vanishing behavior of multizeta values over $\FF_q[t]$ at negative integers. These values are analogs of the classical multizeta values. At negative integers, they are series of products of power sums $S_d(k)$ which are polynomials in $t$. By studying the $t$-valuation of $S_d(s)$ for $s < 0$, we show that multizeta values at negative integers vanish only at trivial zeros. The proof is inspired by the idea of Sheats in the proof of a statement of ``greedy element" by Carlitz.
\end{abstract}


\maketitle

\section{Introduction} \label{sec:intro}

Classical multizeta values (i.e., over $\ZZ$), also known as ``multiple zeta values", are defined as the convergent series
\[
\zeta(\bs) =\sum_{ n_1 > n_2  > \cdots > n_r  \geq 1} \frac{1}{n_1^{s_1}n_2^{s_2} \cdots n_r^{s_r}} \quad \in \RR,
\]
where $\bs=(s_1, \ldots, s_r) \in \ZZ_+^r$ with $s_1 > 1$. We call $r$ the \textit{depth} and $\sum_i s_i$ the \textit{weight} of $\zeta(\bs)$. Here and in the rest of the paper, $\ZZ_+$ is the set of positive integers and $\NN=\ZZ_+ \cup \{0\}$. Multizeta values of depth $1$ are the usual Riemann zeta values. Double zeta values (i.e., $r=2$) were first considered by Euler in 1776 \cite{Euler-1775} in the study of $\zeta(3)$. After a long time of oblivion, multizeta values of higher depth were introduced independently by Hoffman \cite{Hoffman-1992} and Zagier \cite{Zagier-1992} in 1992. During the last three decades, great attention has been drawn to the study of multizeta values because of their appearance in many different contexts, including the absolute Galois group \cite{Goncharov-2000}, periods of mixed Tate motives \cite{Deligne-Goncharov-2005, Goncharov-2005}, knot invariants and calculations of integrals associated to Feynman diagrams in perturbative quantum field theory \cite{Broadhurst-Kreimer-1997}. These various connections with other fields have led to big progresses in the study of classical multizeta values, although some fundamental questions still remain open (see \cite[Preface]{Unpublished:Gil-Fresan}). 

Having learned about the rich interconnections in the classical case, Thakur, in 2002, defined two types of multizeta values over function fields \cite[Sec. 5.10]{Book:Thakur-2004}, one complex valued (generalizing special values of Artin-Weil zeta functions) and the other with values in Laurent series over finite field (generalizing Carlitz zeta values). The first type was completely evaluated in \cite{Book:Thakur-2004} for $\FF_q(t)$ (see \cite{Masri-2006} for a study in the higher genus case). In this paper, we focus on the second type and stick to the rational function field $\FF_q(t)$.

Throughout this paper, $p$ is a prime and $q := p^f$ is a power of $p$. We say an integer is \textit{$q$-even} if it is divisible by $q-1$ and \textit{$q$-odd} otherwise. Let $K := \FF_q(t)$ be the rational function field over the finite field $\FF_q$, $\infty$ be the rational place of $K$ with uniformiser $1/t$ and $K_\infty := \FF_q((1/t))$ be its completion at $\infty$. Let $A := \FF_q[t]$ be the polynomial ring in $t$, $A_+ := \{ \text{monics in $A$} \}$ and $A_{d+}:= \{ \text{monic in $A$ of degree $d$} \}$ for $d \geq 0$. For $d \geq 0$ and $s \in \ZZ$, we define the power sum
\begin{equation}\label{E:powersum}
  S_d(s):=\sum_{a \in A_{d+}} \frac{1}{a^s} \in K.
\end{equation}
Note that $S_d(s) \in A$ if $s < 0$. The \textit{multizeta values} at $\bs \in \ZZ^r$ over $\FF_q[t]$ are defined as
\begin{equation}\label{E:mzv}
    \zeta(\bs):=\sum_{d_1 > \cdots > d_r \geq 0} S_{d_1}(s_1) \cdots S_{d_r}(s_r) \quad \in K_\infty.
\end{equation}
The convergence of $\zeta(\bs)$ at positive integers, i.e., $\bs \in \ZZ_+^r$, is clear from definition of $S_d$. At non-positive integers, it follows from the fact that $S_d(0)=0$ for $d > 0$ and $S_d(s)=0$ for $d \gg 0$ if $s < 0$ (see \cref{sec:main result} for details). At positive integers, the definition above can be restated as
\begin{equation*}
  \zeta(\bs) = \sum_{a_1, a_2, \ldots, a_r} \frac{1}{a_1^{s_1} a_2^{s_2} \cdots a_r^{s_r}} \in K_\infty, 
\end{equation*}
where the sum is over all $a_i \in A_{d_i+}$ with $d_1 > d_2 > \cdots > d_r \geq 0$. Following the classical case, we say $\zeta(\bs)$ is of depth $r$ and weight $\sum_i s_i$. For general introduction of results on function field multizeta values and comparison with the classical case, we refer the reader to the survey papers \cite{Chang-2014, Thakur-2017}. In this paper, $\zeta(\bs)$ is used to denote multizeta values in both the classical and function field cases. It should be clear which one we are referring to from the context. 

A natural question to ask is when $\zeta(\bs)$ vanishes. In classical case, $\zeta(\bs) > 0$ by definition at positive integers with $s_1 > 1$. Treating $s_i$'s as complex variables, the series defining $\zeta(\bs)$ is absolutely convergent in the region
$\{ (s_1, \ldots, s_r) \in \CC^r: \re(s_1 + \cdots + s_j) > j \text{ for } 1 \leq j \leq r \}$ and can be meromorphically continued to $\CC^r$ with singular hyperplanes $\{ s_1=1, s_1+s_2 \in \{2,1,0,-2, -4, -6, \ldots \}, \sum_{i=1}^k s_i \in \ZZ_{\leq k} \text{ for } 3 \leq k \leq r\}$. In particular, all the negative integer points, except when $r=2$ or $s_1+s_2$ odd, lie on these hyperplanes. Moreover, they are points of indeterminacy. See \cite{Furusho-Komori-Matsumoto-Tsumura-2017} and the references mentioned in its ``Introduction" for several different approaches to define and determine the multizeta values at these points.

In function field case, Thakur \cite{Thakur-powersums-2009} showed that $\zeta(\bs) \neq 0$ at positive integers. At negative integers, the vanishing of multizeta values of depth $1$ is completely understood by Goss \cite{Goss-1979}. Its vanishing behavior is quite similar to that of the Riemann zeta values although lacking a functional equation. In this paper, we study the vanishing of $\zeta(\bs)$ at negative integers of higher depth. 

Replacing $a$ by $t^d+\sum_{i=1}^d \theta_i t^{d-i}$ in \eqref{E:powersum}, we can rewrite $S_d(s)$ as a sum of monomials in $t$ for negative $s$, whose sum indices are in $\NN^{d+1}$ satisfying some restrictions. Denote the set of these indices as $U_d(-s)$. Our main result (restated as Theorem \ref{T:modest}) gives an explicit description of the $t$-valuation of $S_d(s)$ in terms of elements in $U_d(-s)$. 

\begin{thm} \label{T:main result}
  Assume $U_d(-s) \neq \emptyset$, then there is a unique monomial in the sum $S_d(s)$ acheving the lowest degree. Moreover, this term correpsonds to the element in $U_d(-s)$ whose reverse is lexicographically the largest.
\end{thm}

This result implies monotonicity of the $t$-valuation of $S_d(s)$ with respect to $d$, using which we completely solve the vanishing of $\zeta(\bs)$ at negative integers (stated as Theorem \ref{T:vanishing at negative} later). See \cref{subsec:trivial} for definition of ``trivial zero".

\begin{thm} \label{T:corollary}
  At negative integers, $\zeta(\bs)$ of depth at least 2 only vanishes at trivial zeros.
\end{thm}

Here is the outline of the paper. In \cref{sec:main result}, we study the behavior of $S_d(s)$ at negative $s$ in detail and discuss how our main result implies Theorem \ref{T:corollary}. \cref{sec:proof} gives the proof of Theorem \ref{T:main result}.

\begin{ack}
  The results of this paper are a part of the author's Ph.D. thesis at the University of Rochester. The author would like to express her sincere gratitude for her advisor, Prof. Dinesh Thakur, for suggesting this problem and for all of his guidance and encouragement.
\end{ack}

\section{Main Result} \label{sec:main result}

In this section, we study the vanishing behavior of multizeta values in detail. We continue to use the notations in the previous section. Our main object of study is $S_d(s)$, the building blocks of multizeta values. The reader will see that the vanishing of $\zeta(\bs)$ is really a reflection of properties of $S_d(s)$.

\subsection{Trivial Zeros.} \label{subsec:trivial}

Let $s < 0$. We first take a closer look at when $S_d(s)$ vanishes. Writing out the coefficients of $a$  in~\eqref{E:powersum}, we get
\begin{align*}
  S_d(s) &= \sum_{\theta_i \in \FF_q} (t^d+\theta_1 t^{d-1}+ \cdots +\theta_d)^{-s} \\
  &= \sum_{\theta_i \in \FF_q}\sum_{\substack{m_0+\cdots +m_d = -s\\ m_i \geq 0}} \binom{-s}{m_0, \ldots, m_d}\, \theta_1^{m_1} \cdots \theta_d^{m_d}\, t^{dm_0+(d-1)m_1+\cdots +m_{d-1}}\\
  &= (-1)^d\sum_{\substack{m_0+\cdots +m_d = -s\\ m_0 \geq 0,\ m_i > 0\ \text{$q$-even for }\, i > 0}} \binom{-s}{m_0, \ldots, m_d}\, t^{dm_0+(d-1)m_1+\cdots +m_{d-1}}\\
  &= (-1)^d\sum_{\substack{\bigoplus_{i=0}^d m_i= -s\\ m_0 \geq 0,\ m_i > 0\ \text{$q$-even for }\, i > 0}} \binom{-s}{m_0, \ldots, m_d}\, t^{dm_0+(d-1)m_1+\cdots +m_{d-1}}, \numberthis \label{E:Sd(s)}
\end{align*}
where $\bigoplus_{i=0}^d m_i$ denotes sum $\sum_{i=0}^d m_i$ with no carry over of digits base $p$. The third equality comes from exchanging the two sum indices and the fact that $\sum_{\theta \in \FF_q} \theta^k=-1$ if $k$ is a positive multiple of $q-1$ and 0 otherwise. The last equality follows from Lucas' theorem. 

For $k>0$ and $d \geq 0$, let
\[
U_d(k) := \{ (m_0, \ldots, m_d) \in \NN^{d+1}: k = \textstyle\bigoplus_{i=0}^d m_i \text{ and } m_i > 0 \text{ is $q$-even for $1 \leq i \leq d$}  \}.
\]
Let $\cP(n)$ be the multiset of $p$-powers adding up to $n$ with no carry over in base $p$. More precisely, if $n=\sum_{i=0}^k a_i p^i$ with $0 \leq a_i < p$, $\cP(n) := \{ \{p^i\}_{a_i}: 0 \leq i \leq k \}$, where $\{m\}_{k}$ denotes the sequence $m, \ldots, m$ with $m$ repeated $k$ times. Then the condition $k = \bigoplus_{i=0}^d m_i$ is equivalent to
\begin{equation} \label{E:no carry over}
  \cP(k) = \bigsqcup_{i=0}^d \cP(m_i).
\end{equation}
Note that $U_d(-s)$ is the set of sum indices in \eqref{E:Sd(s)}. Clearly, $S_d(s)$ vanishes if $U_d(-s) = \emptyset$. In \cite{Carlitz-1948}, Carlitz claimed without proof that the converse also holds. More precisely, he asserted that if $U_d(-s) \neq \emptyset$, the term $t^{dm_0+(d-1)m_1+\cdots +m_{d-1}}$ with $(m_0, \ldots, m_d)$ lexicographically largest among sum indices attains the unique maximal degree. Such $(m_0, \ldots, m_d)$ is called \textit{greedy}. This claim was not proved until 50 years later. Diaz-Vargas \cite{Diaz-Vargas-1996} gave a proof for the case $q=p$ and a general proof for any $q$ is given by Sheats \cite{Sheats-1998}.

\begin{thm}[Calitz, Diaz-Vargas, Sheats] \label{T:greedy}
  For $s< 0$, $S_d(s) \neq 0$ if and only if $U_d(-s) \neq \emptyset$. Moreover, if $U_d(-s) \neq \emptyset$, the summand in $S_d(s)$ corresponding to the greedy element achieves the unique maximal degree.
\end{thm}

B\"{o}eckle pointed out that with some results in \cite{Sheats-1998}, one gets a more straightforward criterion when $S_d(s)$ vanishes. 
\begin{defn}
  For $k \in \ZZ_+$ with base $q$ expansion $k=a_0 + a_1 q + \cdots + a_n q^n$, let $l(k)=\sum a_i$ be the sum of base $q$ digits of $k$. Recall that $q=p^f$. Define
  \[
  L_k:= \min_{i=0,\ldots,f-1} \left\{ \frac{l(kp^i)}{(q-1)} \right\}.
  \]
  We note that since $k \equiv l(k) \smod{q-1}$, $L_k$ is an integer if and only if $(q-1)|k$, i.e., $k$ is $q$-even.
\end{defn}

\begin{prop}[{\cite[Thm. 1.2(a)]{Bockle-2013}}]
\label{P:vanishcondition}
  For $s$ negative, $S_d(s) = 0 \Leftrightarrow d > L_{-s}$.
\end{prop}

For reader's convenience, we provide a proof of the result above. For $d \geq 0$ and $k > 0$, let
\[
V_d(k):=\{ (m_0, \ldots, m_d) \in U_d(k) : m_0 >0 \}.
\]
The proposition follows from the following lemma of Sheats. We note that the notations and expression of the lemma are slightly different from those in Sheats' paper, but one can check that they are equivalent.

\begin{lemma}[{\cite[Prop. 4.3(a)]{Sheats-1998}}] \label{L:vanishing of V_d(k)}
  $V_d(k) = \emptyset \Leftrightarrow d \geq L_k.$
\end{lemma}

\begin{proof}[Proof of Proposition~\ref{P:vanishcondition}:]
  By Theorem \ref{T:greedy}, it is enough to show that $U_d(k) = \emptyset$ iff $d > L_k$. We break it up into two cases.

  If $k$ is $q$-even, $U_d(k)=V_d(k) \cup \{ (0, m_1, \ldots, m_d)\ \big|\ (m_1, \ldots, m_d) \in V_{d-1}(k) \}$. $U_d(k)=\emptyset$ iff $V_d(k)=V_{d-1}(k)=\emptyset$, i.e., $d-1 \geq L_k$ by Lemma \ref{L:vanishing of V_d(k)}. Since $L_k$ is an integer, $d-1 \geq L_k$ $ \Leftrightarrow d > L_k$.

  If $k$ is $q$-odd, then $U_d(k)=V_d(k)$. Thus $U_d(k)=\emptyset$ iff $d \geq L_k$ by Lemma \ref{L:vanishing of V_d(k)}. As $L_k$ is not an integer in this case, $d \geq L_k \Leftrightarrow d > L_k$.
\end{proof}

Note that in \eqref{E:mzv}, the least $d$ appearing in $S_d(s_i)$ is $r-i$. Thus, if $r-i>L_{-s_i}$, all terms in the sum vanishes and so does the multizeta value. With this observation, we define

\begin{defn} \label{D:trivialzero}
  Let $r>1$ and $(s_1, \ldots, s_r) \in \ZZ^r_-$ such that $\zeta (s_1, \ldots, s_r)=0$. We call $(s_1, \ldots, s_r)$ a \textit{trivial zero} of $\zeta$ if there exists some $1 \leq i \leq r-1$ such that $r-i > L_{-s_i}$. Otherwise, $(s_1, \ldots, s_r)$ is a \textit{nontrivial zero}.
\end{defn}

\subsection{Existence of Nontrivial Zero.} \label{subsec:nontrivial}

We now investigate nontrivial zeros of $\zeta(\bs)$ where $s_i < 0$. The depth 1 case is completely understood by Goss \cite{Goss-1979}.

\begin{thm}[{Goss, see \cite[Sec. 5.3]{Book:Thakur-2004}}]
  For $s$ negative, $\zeta(s)=0$ if and only if $s$ is $q$-even.
\end{thm}

Note that multizeta values in this case reduce to Carlitz zeta values. The above theorem shows that the behavior of zeros of Carlitz zeta at negative integers is analogous to that of the trivial zeros of the classical Riemann zeta function. However, unlike a direct implication from the functional equation of the Riemann zeta, the vanishing of $\zeta(s)$, without any known functional equations in the function field case, follows from cancellations among monomials. 

The proof of the nonvanishing of $\zeta(s)$ at $q$-odd $s$ \cite[Thm. 5.3.2]{Book:Thakur-2004} showed that there is a unique term of least degree, 1, in the polynomial sum of $\zeta(s)$, which could not be canceled. Similarly, the fact that multizeta values at positive integers never vanish \cite[Thm. 4]{Thakur-powersums-2009} follows from the strict monotonicity in $d$ of the $\infty$-valuation of $S_d(\bs)$. We use the same strategy to show that there is no nontrivial zeros in higher depth case.

\begin{defn}
  $M=(M_0, \ldots, M_d) \in U_d(k)$ is called \textit{modest} if $(M_d, M_{d-1}, \ldots, M_0)$ is lexicographically the largest, i.e., $M_d \geq m_d$ for all $(m_0, \ldots, m_d) \in U_d(k)$, $M_{d-1} \geq m_{d-1}$ for those $(m_0, \ldots, m_d)$ with $m_d=M_d$ and so on. Such element always exists and is unique if $U_d(k) \neq \emptyset$.
\end{defn}

Our main result is the following theorem, which characterises the term in $S_d(s)$ with least degree. Its proof is given in \cref{sec:proof}.

\begin{thm}\label{T:modest}
  Assume $S_d(s) \neq 0$. The term corresponding to the modest element in $U_d(-s)$ attains the unique minimum degree in $t$ among all summands in $S_d(s)$.
\end{thm}

Recall that for $d \leq L_{-s}$, elements in $U_d(-s)$ and summands in (\ref{E:Sd(s)}) are in one-to-one correspondence. Take $(m_0, m_1, \ldots, m_d) \in U_d(-s)$, then its corresponding term in $S_d(s)$ has degree $dm_0+(d-1)m_1+\cdots+m_{d-1}$. Define $\nu_d(s) := v_t(S_d(s))$, where $v_t$ is the $t$-valuation. We have the following corollary.

\begin{cor} \label{C:valuation monotonicity}
  Fix $s < 0$, then
  \[
  \nu_{\lfloor L_{-s} \rfloor}(s) > \nu_{\lfloor L_{-s} \rfloor-1}(s) > \cdots > \nu_1(s) \geq \nu_0(s).
  \]
\end{cor}

\begin{proof}
  Since $\nu_0(s)=v_t(1)=0$ for all $s$, the last inequality is obvious. Assume $0 < d \leq L_{-s}$ and let $\bld{M}=(M_0, \ldots, M_d)$ be the modest element in $U_d(-s)$, then Theorem \ref{T:modest} implies $\nu_d(s)=dM_0+(d-1)M_1+\cdots+M_{d-1}$. Consider $\bld{N}=(M_0, \ldots, M_{d-2}, M_{d-1}+M_d)$, then $\bld{N} \in U_{d-1}(-s)$ and thus $\nu_{d-1}(s) \leq (d-1)M_0+(d-2)M_1+\cdots+M_{d-2} \leq \nu_d(s)$, where the second inequality is equality iff $d=1$ and $M_d=-s$.
\end{proof}

With this result, we finish the discussion of the vanishing of multizeta values of higher depth at negative integers.

\begin{thm} \label{T:vanishing at negative}
  For $\bs=(s_1, \ldots, s_r)$ with $s_i < 0$ and $r > 1$, $\zeta(\bs)=0$ if and only if $\bs$ is a trivial zero.
\end{thm}

\begin{proof}
  It is equivalent to show that $\zeta(\bs) \neq 0$ if $\bs$ is not a trivial zero. In this case, the sum $\zeta(\bs)=\sum_{d_1 > \cdots > d_r \geq 0} S_{d_1}(s_1) \cdots S_{d_r}(s_r)$ is nonempty. In particular, $S_{r-1}(s_1) \cdots S_0(s_r) \neq 0$ and 
  \[v_t(S_{r-1}(s_1) \cdots S_0(s_r))=\sum_{i=1}^r \nu_{r-i}(s_i).
  \]
  For any other term $S_{d_1}(s_1) \cdots S_{d_r}(s_r)$ in the sum, $d_i \geq r-i$ for all $i$ and there exist some $j$ such that $d_j > r-j > 0$. Thus, by Corollary \ref{C:valuation monotonicity},
  \[
  v_t(S_{d_1}(s_1) \cdots S_{d_r}(s_r))=\sum_{i=1}^r \nu_{d_i}(s_i) > v_t(S_{r-1}(s_1) \cdots S_0(s_r)).
  \]
  By strict triangle inequality, $v_t(\zeta(\bs))=v_t(S_{r-1}(s_1) \cdots S_0(s_r))=\sum_{i=1}^r \nu_{r-i}(s_i)$. In particular, $\zeta(\bs) \neq 0$.
\end{proof}

\begin{rmk}
  We note that the same strategy fails in analysing the vanishing of $\zeta(\bs)$ at integers of mixed signs. For both place $t$ and $\infty$, $s$ being positive and negative give opposite monotonicity of the valuation of $S_d(s)$ in $d$. Hence there is no unique term with least valuation in general. For example, let $q=3$, 
  \begin{align*}
      \zeta(-8, 2) &= S_1(-8)S_0(2)+S_2(-8)S_0(2)+S_2(-8)S_1(2)\\
      &= (2t^6+2t^4+2t^2+2)+(t^6+t^4+t^2)+(1)=0
  \end{align*}
  is a ``nontrivial zero" in the sense of Definition \ref{D:trivialzero}. In the sum, $S_1(-8)S_0(2), S_2(-8)S_1(2)$ attain the least valuation at $t$ and $S_1(-8)S_0(2), S_2(-8)S_0(2)$ attain the least valuation at $\infty$.
\end{rmk}

\section{Proof of Theorem \ref{T:modest}} \label{sec:proof}

The proof of Theorem \ref{T:modest} is quite complicated and combinatorial. This is because the two conditions on elements of $U_d(-s)$ are with respect to $p$ and $q$ each while $p$ and $q$ are different in general. Major difficulty of the proof arises from how to track these two conditions simultaneously. 

\subsection{Special case.} \label{subsec:special}

When $q=p$ is a prime, the problem mentioned above disappears and the theorem can be proved in a way similar to the proof of Theorem \ref{T:greedy} for $q=p$ case by Diaz-Vargas. Another simple case, without restriction on $q$, is where $s$ is $q$-even, which follows directly from the result on greedy element. We first prove these two special cases.

\begin{proof}[Proof of Theorem \ref{T:modest} for special cases]
  Let $k=-s$ and $\bld{M}=(M_0, \ldots, M_d) \in U_d(k)$ be the modest element. For $\bld{m} = (m_0, \ldots, m_d)$ $\in U_d(k)$, define
  \[
  \wt(\bld{m}):=dm_0+(d-1)m_1+\cdots+m_{d-1}
  \]
  to be its weight, which equals the degree of its corresponding term in $S_d(s)$. For both cases, we need to show $\bld{M}$ achieves the unique minimum weight.

  (1) $q=p$ is a prime: We show that given any non-modest element $\bld{m}$, one can always adjust it to get another $\bld{m'}$ of smaller weight.  Let $l > 0$ be the largest index such that $M_l > m_l$. Then $M_i=m_i$ for $i>l$ by the choice of $\bld{M}$. Recall that $\cP(n)$ is the multiset of $p$-powers represented by the base $p$ digits of $n$. When $q=p$,  $n$ is $q$-even iff $(q-1)|\#\cP(n)$. We split the discussion into two cases. 
  \begin{enumerate}
      \item[(a)] If $\#\cP(M_l) \leq \#\cP(m_l)$, then there exist some $p^e \in \cP(m_l)$ and $p^{e'} \in \cP(M_l)\backslash\cP(m_l)$ such that $p^e < p^{e'}$. By \eqref{E:no carry over}, $p^{e'} \in \cP(m_{l'})$ for some $l' < l$. Let
      \[
      \bld{m'}=(m_0, \ldots, m_{l'}-p^{e'}+p^e, \ldots, m_l-p^e+p^{e'}, \ldots, m_d),
      \]
      then it is easy to check that $\bld{m'} \in U_d(k)$ and $\wt(\bld{m'}) < \wt(\bld{m})$.
      \item[(b)] If $\#\cP(M_l) > \#\cP(m_l)$, then $\#\cP(M_l) - \#\cP(m_l) \geq q-1$ since both $M_l$ and $m_l$ are $q$-even. Note that $\sum_{i=0}^d \#\cP(M_i)=\sum_{i=0}^d \#\cP(m_i)=\#\cP(k)$. Thus there exists $l'<l$ such that $\#\cP(m_{l'}) - \#\cP(M_{l'}) \geq q-1$. Write $\cP(m_{l'})=P_1 \sqcup P_2$, where $\#P_1 = q-1$, and this implies $m_{l'}=n_1 \oplus n_2$ with $n_1$ $q$-even. If $l'>0$, then $n_2 > 0$ and is also $q-even$. Consider
      \[
      \bld{m'}=(m_0, \ldots, m_{l'}-n_1, \ldots, m_l+n_1, \ldots, m_d),
      \]
      then $\bld{m'} \in U_d(k)$ and $\wt(\bld{m'}) < \wt(\bld{m})$.
  \end{enumerate} 

  (2) $s$ is $q$-even: Recall that
  \[
  V_d(k)=\{(m_0, \ldots, m_d) \in U_d(k) : m_0 > 0 \}.
  \]
  In this case, $\bld{M} \in U_d(k) \setminus V_d(k)$ since otherwise $(0, M_1, \ldots, M_d+M_0)$ is also contained in $U_d(k)$ whose reverse is lexicographically larger. Similar argument shows that $\bld{m} \in U_d(k) \setminus V_d(k)$ if $\bld{m}$ is of minimum weight. Consider the bijective map 
  \[
  \varphi: (0, m_1, \ldots, m_d) \mapsto (m_d, \ldots, m_1)
  \]
  between $U_d(k) \setminus V_d(k)$ and $V_{d-1}(k)$. Note that $k=\sum_i m_i$, thus, for $\bld{m} \in U_d(k) \setminus V_d(k)$,
  \[
  \wt(\bld{m})=(d-1)m_1+\cdots+m_{d-1}=(d-1)k-\wt(\varphi(\bld{m})).
  \]
  $\wt(\bld{m})$ being minimum indicates that $\varphi(\bld{m})=(m_d, \ldots, m_1)$ achieves the largest weight in $V_{d-1}(k)$. By Theorem \ref{T:greedy}, $\varphi(\bld{m})$ has to be the greedy element in $U_{d-1}(k)$. This implies that the reverse of $\bld{m}$ is lexicographically the largest in $U_d(k) \setminus V_d(k)$, hence $\bld{m}=\bld{M}$.
\end{proof}

\subsection{General case.} \label{subsec:general}

Our proof for general case is inspired by Sheats' proof \cite{Sheats-1998} of Theorem \ref{T:greedy} on greedy element. We prove by contradiction. Roughly speaking, assuming there exists a tuple not modest in $U_d(-s)$ gives a term of lowest degree in $S_d(s)$, we construct another term with smaller degree.

We fix a prime power $q=p^f$. In this section, $\bar{x}$ denotes a column vector of length $f$, where $x$ is either an English or Greek letter, with or without subscript. If not mentioning explicitly, its entires are denoted as $x_i$ with $0 \leq i < f$, e.g., $\bar{u}=[u_0, u_1, \ldots, u_{f-1}]^t$. Note that the subscripts start from $0$. The zero vector is denoted as $\bar{0}$.

\subsubsection{Set up and preliminaries.} \label{subsub:setup}

Before the proof, we change to a different notation for easy expression. A $d$-tuple $(X_1, \ldots, X_d) \in \NN^d$ is said to be a \textit{composition} of $N$ if $N=\sum_{i=1}^d X_i$. For $d > 0$ and $ N \in \ZZ_+$, let
\begin{align*}
  W_d(N) &= \{ (X_1, X_2, \ldots, X_d) \in \NN^d : (X_d, X_{d-1}, \ldots, X_1) \in U_{d-1}(N) \}\\
  &= \{ (X_1, X_2, \ldots, X_d) \in \NN^d : N = \bigoplus_{i=1}^d X_i, \text{ and } X_i > 0 \text{ is $q$-even for } i < d \}.
\end{align*}
In this new set up, the modest element in $U_{d-1}(N)$ corresponds to be the lexicographically largest composition in $W_d(N)$, which we again call it modest.

\begin{defn}
  Let $\bld{X}=(X_1, \ldots, X_d) \in W_d(N)$. Define its \textit{weight}, denoted as $\wt(\bld{X})$, by
  \[
  \wt(\bld{X})=X_1+2X_2+\cdots+dX_d.
  \]
  Any composition $\bld{X}$ achieving the minimum weight in $W_d(N)$ is called \textit{optimal}. 
\end{defn}

One can check that Theorem \ref{T:modest} is equivalent to the following.

\begin{thm}\label{T:modest'} 
  For $W_d(N) \neq \emptyset$, the modest composition is the only optimal composition. 
\end{thm}

\begin{rmk} \label{R:trivial case}
  The theorem holds for $d=1$ trivially since $W_1(N)=\{(N)\}$ has only one composition. For $d=2$, $\wt(\bld{X})=2N-X_1$ for any $\bld{X} \in W_2(N)$ and hence the modest composition is the only optimal element.
\end{rmk}

The following proposition consists of some observations on how to get new modest or optimal compositions from old ones.

\begin{prop} \label{P:basicprop}
  Suppose $W_d(N) \neq \emptyset$. $\bld{X}=(X_1, \ldots, X_d)$ is the modest composition in $W_d(N)$. Then
  \begin{itemize}
    \item[(i)] $(X_1, X_2, \ldots, X_{d-1})$ is the modest composition in $W_{d-1}(N-X_d)$;
    \item[(ii)] $(X_2, X_3, \ldots, X_d)$ is the modest composition in $W_{d-1}(N-X_1)$;
    \item[(iii)] for any $n \geq 0$, $(p^nX_1, \ldots, p^nX_n)$ is the modest composition in $W_d(p^nN)$.
  \end{itemize}
  These three statements remain true when replacing ``the modest composition" with ``an optimal composition".
\end{prop}

\begin{proof}
  (i) and (ii) are obvious from definition in each case. To show (iii), we observe that all $p$-powers in $\cP(p^nN)$ are divisible by $p^n$. Thus, for $(Y_i) \in W_d(p^nN)$, $p^n \mid Y_i$ for all $i$ since $\cP(Y_i) \subset \cP(p^nN)$. Moreover, $(Y_i) \mapsto (p^{-n}Y_i)$ gives a 1-to-1 correspondence between compositions in $W_d(p^nN)$ and $W_d(N)$. (iii) follows from this observation easily in both cases.
\end{proof}

Given base $p$ expansion $n=\sum_{j \geq 0} a_jp^j$, we define $\Gamma(n) \in \NN^f$ to be the column vector $[\mu_0, \ldots, \mu_{f-1}]^t$, where 
\[
\mu_i=\sum_{j \equiv i \smod{f}} a_j.
\]
Let $\bar{\psi}_0 := [1, p, \ldots, p^{f-1}]^t$, then
\[
\inner{\bar{\psi}_0}{\Gamma(n)}=\mu_0+\cdots+p^{f-1}\mu_{f-1}
\]
is the sum of base $q$ digits of $N$. In particular, $n$ is $q$-even iff $(q-1) \mid \inner{\bar{\psi}_0}{\Gamma(n)}$. Then $\bld{X} \in W_d(N)$ if and only if
\begin{itemize}
  \item[(1)] $\Gamma(N)=\Gamma(X_1)+\Gamma(X_2)+\cdots+\Gamma(X_d)$,
  \item[(2)] for $1 \leq i \leq (d-1)$, $(q-1) \mid \inner{\bar{\psi}_0}{\Gamma(X_i)} \neq 0$. 
\end{itemize}
For a composition $\bld{X}=(X_1, \ldots, X_d)$ of $N$, define $\Gamma(\bld{X})$ to be the $f \times d$ matrix with columns $\Gamma(X_1), \ldots, \Gamma(X_d)$.

\begin{exmp} 
  Let $q=9$ and $N=131$. In base 3, $N=11212_3$. Thus $\Gamma(N)=[5, 2]^t$. For any $\bld{X} \in W_2(N)$, $\Gamma(\bld{X})$ is one of the two matrices: $\begin{bmatrix}5 & 0\\ 1 & 1\end{bmatrix}, \begin{bmatrix}2 & 3\\ 2 & 0\end{bmatrix}$. $(128, 3)=(11202_3, 10_3)$, $(104, 27)=(10212_3, 1000_3) \in W_2(N)$ correspond to the first one, and the rest correspond to the second one. 
\end{exmp}

In the example above, we give a partition of compositions in $W_d(N)$ with respect to matrix representation. Given $B \in \Mat_{f \times d}(\NN)$, define
\[
W^B_d(N):=\{\bld{X} \in W_d(N): \Gamma(\bld{X})=B\}.
\]
We call $B$ a \textit{valid matrix} of $W_d(N)$ if $W^B_d(N) \neq \emptyset$. Let $B_1, \ldots, B_d$ be columns of $B$, then $B$ is valid if and only if 
\begin{itemize}
  \item[(1)] $\Gamma(N)=B_1+\cdots+B_d$,
  \item[(2)] for $1 \leq i \leq (d-1)$, $(q-1) \mid \inner{\bar{\psi}_0}{B_i} \neq 0$. 
\end{itemize}
For $n > 0$, denote $\tau(n)$ the nonincreasing sequence of $p$-powers in $\cP(n)$ and $\tau_k(n)$ be its subsequence consisting of those $p^i$ with $i \equiv k \smod{f}$ for $0 \leq k < f$. 

\begin{exmp} 
  Take $q=9$ and $N=131=11212_3$. Then 
  \[
  \tau(N)=(3^4, 3^3, 3^2, 3^2, 3^1, 3^0, 3^0),~ \tau_0(N)=(3^4, 3^2, 3^2, 3^0, 3^0), ~\tau_1(N)=(3^3, 3^1).
  \]
\end{exmp}

Given $\bld{X} \in W_d(N)$, $\tau_k(X_i)$'s give a partition of $p$-powers in $\tau_k(N)$ for each $k$. We call $\bld{X}$ is \textit{$\tau$-monotonic} if the sequence $\tau_k(N)$  is the concatenation of the subsequences $\tau_k(X_1), \tau_k(X_2), \ldots, \tau_k(X_d)$ for all $0 \leq k < f$. Note that there is a unique $\tau$-monotonic composition in $W^B_d(N)$ for each valid $B$.

\begin{lemma}
  Suppose $B$ is a valid matrix of $W_d(N)$, then the $\tau$-monotonic composition with respect to $B$ is lexicographically the largest and acheives the unique minimum weight in $W_d^B(N)$. In particular, both modest and optimal compositions are $\tau$-monotonic.
\end{lemma}

\begin{proof}
  Take $\bld{X}=(X_1, \ldots, X_d) \in W_d^B(N)$ which is not $\tau$-monotonic. Then there exist some $k, i, j, m, n$ such that $i < j, m < n$, with $p^m \in \tau_k(X_i), p^n \in \tau_k(X_j)$. Consider the composition $\bld{Y}=(X_1, \ldots, X_i-p^m+p^n, \ldots, X_j-p^n+p^m, \ldots, X_d)$. Then $\bld{Y} \in W_d^B(N)$ since $m \equiv n \equiv k \smod{f}$. Clearly, $\bld{Y}$ is lexicographically larger than $\bld{X}$. Easy computation shows that $\wt(\bld{Y})=\wt(\bld{X})-(j-i)(p^n-p^m) < \wt(\bld{X})$.
\end{proof}

Define
\[
\fJ := \{\Gamma(n): n > 0 \text{ is $q$-even} \}.
\]
Given $B=[B_1, \ldots, B_d]$ an $f \times d$ matrix with columns $B_i$, the conditions for $B$ being valid for $W_d(N)$ can be translated as
\begin{itemize}
  \item[(1)] $\Gamma(N)=B_1+\cdots+B_d$,
  \item[(2)] $B_i \in \fJ$ for $1 \leq i \leq (d-1)$. 
\end{itemize}
We follow Sheats' discussion in \cite{Sheats-1998} to give a characterization of vectors in $\fJ$. Let $\ebar_0, \ldots, \ebar_{f-1}$ be the standard basis of $\RR^f$, i.e., $[\ebar_0, \ldots, \ebar_{f-1}]=I$, the identity matrix. Define matrix $E=[E_0, E_1, \ldots, E_{f-1}]$ with columns
\[
E_i := p\ebar_{i-1}-\ebar_i.
\]
Here and from now on, subcripts which should range from $0$ to $f-1$ are evaluated modulo $f$, e.g. $\ebar_{-1}=\ebar_{f-1}$ and $E_0=p\ebar_{f-1}-\ebar_0$. Given vectors $\bar{u}$ and $\bar{v}=E\bar{u}$, we have, for all $i$
\[
v_i=pu_{i+1}-u_i.
\]
Let $R = [\ebar_1, \ebar_2, \ldots, \ebar_{f-1}, \ebar_0]$ be the permutation matrix such that $R\ebar_i=\ebar_{i+1}$. Then $R^f=I$ and $\inner{R\bar{u}}{R\bar{v}}=\langle \bar{u}, \bar{v} \rangle$ for any $\bar{u}$ and $\bar{v}$. Recall that $\bar{\psi}_0= [1, p, \ldots, p^{f-1}]^t$. Define 
\[
\bar{\psi}_i := R^i \bar{\psi}_0=[p^{f-i}, \ldots, p^{f-1}, 1, \ldots, p^{f-1-i}]^t
\]
for $1 \leq i < f$. Then
\[
\inner{\bar{\psi}_i}{E_j}=\begin{cases} q-1 \quad &\text{if } i=j\\ 0 \quad &\text{otherwise,} \end{cases}
\]
which implies
\[
E^{-1}=(q-1)^{-1}[\bar{\psi}_0, \bar{\psi}_1, \ldots, \bar{\psi}_{f-1}]^t.
\]
Given two vectors $\bar{u}$ and $\bar{v}$, we denote $\bar{u} \geq \bar{v}$ if $u_i \geq v_i$ for all $i$, $\bar{u} > \bar{v}$ if $\bar{u} \geq \bar{v}$ and $u_i > v_i$ for some $i$ and $\bar{u} \gg \bar{v}$ if  $u_i > v_i$ for all $i$.

\begin{lemma} \label{L:lemma1}
  Let $\bar{u}=E\bar{a}$ and $\bar{v}=E\bar{b}$, then
  \begin{itemize}
    \item[(i)] $\bar{u} > \bar{v} \Rightarrow \bar{a} \gg \bar{b}$. In particular, if $\bar{u} > \bar{0}$, then $\bar{a} \gg \bar{0}$.
    \item[(ii)] Let $\bar{1}=[1, \ldots, 1]^t$. If $\bar{0} < \bar{u} < (p-1) \bar{1}$, then $\bar{0} \ll \bar{a} \ll \bar{1}$.
  \end{itemize}
\end{lemma}

\begin{proof}
  $\bar{a}-\bar{b}=E^{-1}(\bar{u}-\bar{v})$. Since all components of $E^{-1}$ are positive, $\bar{u}-\bar{v}>\bar{0}$ implies $\bar{a}-\bar{b} \gg \bar{0}$. This proves (i). (ii) is a direct application of (i) as $[p-1, \ldots, p-1]^t=E[1, \ldots, 1]^t$.
\end{proof}

Take a positive integer $n$. Let $E\bar{\alpha}=\Gamma(n)$, then we have, for each $i$
\[
\alpha_i=(q-1)^{-1}\inner{\bar{\psi}_0}{\Gamma(p^{f-i}n)}
\]
since $R\Gamma(n)=\Gamma(pn)$ and $\inner{\bar{\psi}_i}{\Gamma(n)}=\inner{R^i\bar{\psi}_0}{\Gamma(n)}=\inner{\bar{\psi}_0}{R^{f-i}\Gamma(n)}$. In particular,
\[
\bar{\alpha} \in \ZZ^f \Leftrightarrow n \mbox{ is $q$-even}.
\]
The above discussion can be rephrased as following.

\begin{prop}[{\cite[Lem. 3.4]{Sheats-1998}}]\label{P:integral}
  $\fJ=(E\ZZ^f)\cap (\NN^f\backslash\{\bar{0}\})$.
\end{prop}

\subsubsection{Criterion for $W_d(N) \neq \emptyset$.} \label{subsub:IJ} 

For $d > 0$, define
\[
I_d := \{ \Gamma(n): \exists\ \bar{v}_1, \ldots, \bar{v}_{d-1} \in \fJ \mbox{ such that } \Gamma(n) > \bar{v}_1+\cdots+\bar{v}_{d-1}\},
\]
\[
\quad J_d := \fJ \cap (I_d \backslash I_{d+1}).
\]
For $d=0$, we set $J_0 = \emptyset$. By definition, $J_d$ consists of those $\Gamma(n)$ such that $n$ can be written as a sum of $d$ many, but not $d+1$ many, positive $q$-even numbers without carry over in base $p$. Then 
\begin{equation} \label{E:IJ}
W_d(N) \neq \emptyset \Leftrightarrow \Gamma(N) \in J_{d-1} \cup I_d.
\end{equation}
The next proposition by Sheats characterizes elements in $I_m$ and $J_m$.

\begin{prop}[{\cite[Prop. 4.3]{Sheats-1998}}] \label{P:IJ}
  For $m \geq 1$,
  \begin{itemize}
    \item[(i)] $I_m=\{ E\bar{x} \in \NN^f\backslash\{\bar{0}\}: \bar{x} \in \RR^f \textup{ and } \min_{0 \leq i < f}(x_i)>m-1\}$,
    \item[(ii)] $J_m=\{ E\bar{a} \in \NN^f\backslash\{\bar{0}\}: \bar{a} \in \RR^f \textup{ and } \min_{0 \leq i < f}(a_i)=m\}$.
  \end{itemize}
\end{prop}

With \eqref{E:IJ}, it implies the following result which is indeed equivalent to Proposition \ref{P:vanishcondition}.

\begin{cor} \label{C:nonvanishing condition}
  Let $\Gamma(N)=E\bar{\alpha}$, then $W_d(N) \neq \emptyset$ \text{iff} $\min_{0 \leq i < f}(\alpha_i) \geq d-1$. \null\hfill\qedsymbol
\end{cor}

\subsubsection{Modest/optimal composition.} \label{subsub:modest/optimal}

The following results give estimation on components of the modest and optimal compositions.

\begin{prop}\label{P:estimation on M/O}
  Let $\bld{X}=(X_1, \ldots, X_d) \in W_d(N)$ be modest or optimal. Then $\Gamma(X_i) \in J_1$ for $2 \leq i \leq d-1$, $X_d=0$ if $N$ is $q$-even or $\Gamma(X_d) \in I_1\backslash I_2$ if $N$ is $q$-odd.
\end{prop}

\begin{proof}
  We prove by contrapositive. 
  
  $N$ is $q$-even: If $X_d \neq 0$, then $\bld{X}$ is neither modest nor optimal from the discussion in \cref{subsec:special}. If $\Gamma(X_i)=\bar{v}_1+\bar{v}_2$ for some $2 \leq i \leq d-1$ and vectors $\bar{v}_1, \bar{v}_2 \in \fJ$, define 
  \[
  \bld{Y} := ( X_1+a_1, \ldots, X_{i-1}, a_2, X_i, \ldots, X_d ).
  \]
  Since $\bar{v}_i \in \fJ$, both $a_i$'s are $q$-even. The sum of entries in $\bld{Y}$ has no carry over in base $p$. So $\bld{Y} \in W_d(N)$. Moreover, $\bld{Y}$ is lexicographically larger than $\bld{X}$, and $\wt(\bld{Y})=\wt(\bld{X})-(i-1)a_1<\wt(\bld{X})$.
  
  $N$ is $q$-odd: Suppose $\Gamma(X_d)=\bar{w}_1+\bar{w}_2$ with $\bar{w}_1 \in \fJ, \bar{w}_2 \in I_1$. We have $X_d=b_1 \oplus b_2$ with $\Gamma(b_i)=\bar{w}_i$ and $b_1$ $q$-even. Define $\bld{Y} := (X_1+b_1, X_2, \ldots, X_{d-1}, b_2).$ Then similar argument as above shows that $\bld{Y} \in W_d(N)$, $\bld{Y}$ is lexicographically larger than $\bld{X}$ and $\wt(\bld{Y})<\wt(\bld{X})$.
\end{proof}

Take $N \in \ZZ_+$, let $\bar{u}=\Gamma(N)$ and $\bar{\beta}=E^{-1}\bar{u}$.

\begin{lemma} \label{L:main lemma}
  Let $\bar{v}=E\bar{\alpha} \in \NN^f$ with $\bar{0} < \bar{v} < \bar{u}$. Suppose $\min_{0 \leq i < f}(\floor{\beta_i}-\ceil{\alpha_i}) = k$ for some  $k \in \NN$, then there exists some $\bar{w} \in \fJ$ with $\bar{v} \leq \bar{w} \leq \bar{u}$ and $\bar{u}-\bar{w} \in J_k \cup I_{k+1}$.
\end{lemma}

\begin{proof} 
  To find such an $\bar{w}$ is equivalent to find a $\bar{\gamma}$ with $\bar{w}=E\bar{\gamma}$. Recall that if $\bar{x}=E\bar{a}$, $x_i=pa_{i+1}-a_i$. By Propositions \ref{P:integral} and \ref{P:IJ}, we get the following conditions on $\bar{\gamma}$:
  \begin{itemize}
    \item[(1)] $v_i \leq p\gamma_{i+1}-\gamma_i \leq u_i$,
    \item[(2)] $\gamma_i \in \ZZ$ and $\min_{0 \leq j < f}(\beta_j-\gamma_j) \geq k$.
  \end{itemize}
  To construct $\bar{\gamma}$, take an $l$ such that $\floor{\beta_l}-\ceil{\alpha_l}=k$. Let $\gamma_l=\ceil{\alpha_l}$. For $i=l-1, l-2, \ldots, l-f+1$, define inductively
  \[
  \gamma_i=\min(\floor{\beta_i}-k, p\gamma_{i+1}-v_i).
  \]
  Condition (2) holds automatically by the construction of $\gamma_i$. The construction also implies $v_i \leq p\gamma_{i+1}-\gamma_i$ for $i \neq l$. To prove it for $i=l$, we first show that $\gamma_i \geq \ceil{\alpha_i}$ for all $i$. By definition, $\gamma_l=\ceil{\alpha_l}$. We prove the rest by backwards induction. Suppose $\gamma_{i+1} \geq \ceil{\alpha_{i+1}}$. If $\gamma_i=\floor{\beta_i}-k$, then clearly $\gamma_i \geq \ceil{\alpha_i}$ since $\floor{\beta_i}-\ceil{\alpha_i} \geq k$; otherwise,
  \begin{align*}
    \gamma_i & = p\gamma_{i+1}-v_i\\
    & = p\gamma_{i+1}-p\alpha_{i+1}+\alpha_i \geq \alpha_i.
  \end{align*}
  Since $\gamma_i$ is an integer, we have $\gamma_i \geq \ceil{\alpha_i}$. Now we have $v_l \leq p\gamma_{l+1}-\gamma_l$ since
  \[
    p\gamma_{l+1}-v_l = p\gamma_{l+1}-p\alpha_{l+1}+\alpha_l \geq \ceil{\alpha_l}=\gamma_i.
  \]
  Next we show $p\gamma_{i+1}-\gamma_i \leq u_i$. For $i=l$,
  \begin{align*}
    u_l-(p\gamma_{l+1}-\gamma_l) & = p\beta_{l+1}-\beta_l-(p\gamma_{l+1}-\gamma_l)\\
    & = p(\beta_{l+1}-\gamma_{l+1})-(\beta_l-\gamma_l)\\
    & = p(\beta_{l+1}-\gamma_{l+1})-(\beta_l-\ceil{\alpha_l})\\
    & > p(\beta_{l+1}-\gamma_{l+1})-k-1 \geq -1,
  \end{align*}
  where the last inequality comes from that $\beta_{l+1}-\gamma_{l+1} \geq k$ and $p \geq 2$. Since the left-hand side is an integer, we have
  \[
  u_l-(p\gamma_{l+1}-\gamma_l) \geq 0.
  \]
  Now let $i \neq l$. If $\gamma_i=p\gamma_{i+1}-v_i$, then $p\gamma_{i+1}-\gamma_i=v_i \leq u_i$; otherwise, $\gamma_i=\floor{\beta_i}-k$, then a similar computation as in the $i=l$ case shows $p\gamma_{i+1}-\gamma_i \leq u_i$.
\end{proof}

\begin{prop} \label{P:weight estimation for modest/optimal elt}
  Take $N$ with base $p$-expansion $N=\sum_{i=0}^n a_ip^i$, where $a_n \neq 0$. Suppose $W_d(N) \neq \emptyset$. Let $\bld{X}=(X_1, \ldots, X_d) \in W_d(N)$ be modest or optimal, then
  \begin{itemize}
    \item[(i)] $X_1 \geq a_np^n$. In particular, $X_1>N/2$.
    \item[(ii)] $N \leq \wt(\bld{X}) < 2N$.
    \item[(iii)] $W_d(N-X_1)=\emptyset$ if $d \geq 2$. In particular, by \eqref{E:IJ}, $\Gamma(N-X_1) \not\in I_d$.
\end{itemize}
\end{prop}

\begin{proof} 
  We prove each case separately.

  $\bld{X}$ modest: Let $\bar{u}=\Gamma(N)$ and $\bar{\alpha}=E^{-1}\bar{u}$. By Corollary \ref{C:nonvanishing condition}, $\min_i(\floor{\alpha_i}) = m \geq d-1.$ Let $k=n \smod{f}$ and $\bar{\beta}=E^{-1}(a_n\bar{e}_k)$. Lemma \ref{L:lemma1} implies $\ceil{\beta_i} = 1$ for each $i$. By Lemma \ref{L:main lemma}, we can extend $\bar{v}=a_n\bar{e}_k$ to some $\bar{w}_1 \in \fJ$ with $\bar{u}-\bar{w}_1 \in J_{m-1} \cup I_{m}$. In particular, we can write $\bar{u}-\bar{w}_1$ as $\bar{u}-\bar{w}_1=\bar{w}_2+\cdots+\bar{w}_{d-1}+\bar{w}_d$, where $\bar{w}_i \in \fJ$ for $2 \leq i \leq d-1$ and $\bar{w}_d \in \NN^d$. Take $B=[\bar{w}_1, \ldots, \bar{w}_d],$ then $B$ is a valid matrix, i.e., $W_d^B(N) \neq \emptyset$. Let $\bld{Y}$ be the $\tau$-monotonic element in $W_d^B(N)$, then $Y_1 \geq a_np^n$ since $\bar{w}_1 \geq a_n\bar{e}_k$. $X$ is modest so $X_1 \geq Y_1 \geq a_np^n$. This proves (i).

  We prove (ii) by induction on $d$. For $d=1$, $\bld{X}=(N)$ and $\wt(\bld{X})=N$. Suppose (ii) holds for $d-1$. By Proposition \ref{P:basicprop} (ii), $\bld{Y}=(X_2, \ldots, X_d)$ is modest in $W_{d-1}(N-X_1)$. By induction, $N-X_1 \leq \wt(\bld{Y})<2(N-X_1)<N$. Thus $N \leq \wt(\bld{X})=N+\wt(\bld{Y})<2N$.

  Suppose (iii) fails. Take $(X'_1, \ldots, X'_{d}) \in W_d(N-X_1)$, then $(X_1+X'_1, X'_2, \ldots, X'_{d}) \in W_d(N)$. But this contradicts that $\bld{X}$ is modest. 

  $\bld{X}$ optimal: (ii) holds automatically by the minimum weight property. 
  
  To prove (i), we first show that $X_1 \geq p^n$. If $N$ is $q$-even or $d < 3$, by \cref{subsec:special} and Remark \ref{R:trivial case}, optimal is equivalent to modest, thus (1) holds. We assume $N$ is $q$-odd and $d \geq 3$. Then $X_1+X_2 > p^n$ since otherwise $X_1+X_2 < p^n < N/2$ and $\wt(\bld{X})>N+2(N-X_1-X_2)>2N$. For $\wt(\bld{X})$ being minimal, $X_1 \geq X_2$ which implies $X_1 \geq p^n$. Now suppose $X_1=\sum_{i=0}^nb_ip^i$ with $0 < b_n < a_n$, then $N-X_1>p^n$. Note that $(X_2, \ldots, X_d)$ is optimal in $W_{d-1}(N-X_1)$, thus $X_2 \geq p^n$ and $N-X_d>(b_n+1)p^n$. But by Proposition \ref{P:basicprop} (i), $(X_1, \ldots, X_{d-1}) \in W_{d-1}(N-X_d)$ is optimal and thus modest since $N-X_d$ is $q$-even. In particular, $X_1 \geq (b_n+1)p^n$. Contradiction.

  At last, we show (iii) holds. If not, let $(X'_1, \ldots, X'_d) \in W_d(N-X_1)$ be optimal. Then $X'_1>(N-X_1)/2$ by (i). $(X'_2, \ldots, X'_d)$ being optimal in $W_{d-1}(N-X_1-X'_1)$, $\wt(X'_2, \ldots, X'_d)<2(N-X_1-X_1')<N-X_1$. Let $\bld{Y}=(X_1+X'_1, X'_2, \ldots, X'_d)$, then $\bld{Y} \in W_d(N)$ and $\wt(\bld{Y})=N+\wt(X'_2, \ldots, X'_d)<2N-X_1$. However, $\wt(\bld{X})=N+\wt(X_2, \ldots, X_d)\geq 2N-X_1$. Contradiction.
\end{proof}

\subsubsection{Constructing composition of smaller weight.} \label{subsub:constructing}

Suppose Theorem \ref{T:modest'} fails. Then $q=p^f$ with $f > 1$. Take the least $d$ and some $N$ such that there exist $\bld{M}=(M_i), \bld{O}=(O_i) \in W_d(N)$ with $\bld{M}$ modest, $\bld{O}$ optimal and $\bld{M} \neq \bld{O}$. Then $N$ is $q$-odd, $d \geq 3$ and $M_1 > O_1$ by \cref{subsec:special}, Remark \ref{R:trivial case} and Proposition \ref{P:basicprop} (ii) respectively. Let $p^a$ be the largest $p$-power in $\cP(M_1) \setminus \cP(O_1)$. By Proposition \ref{P:basicprop} (iii), we may assume $a \equiv f-1 \smod{f}$. Let 
\begin{align*}
  \bar{u}&:=\Gamma(N), & \bar{x}&:=\Gamma(M_1), & \bar{y}&:=\Gamma(O_1),\\ \bar{\eta}&:=E^{-1}\bar{u}, & \bar{\alpha}&:=E^{-1}\bar{x},&
  \bar{\beta}&:=E^{-1}\bar{y}.
\end{align*}
By our construction, $x_{f-1} > y_{f-1}$ since both $\bld{M}$ and $\bld{O}$ are $\tau$-monotonic. Define
\begin{equation*} 
  \bar{v}=[v_i]^t:=[\min(x_i, y_i)]^t, \quad \bar{w}=[w_i]^t:=\bar{y}-\bar{v}. 
\end{equation*}
Then $w_{f-1}=0$. Note that $\bar{w} > \bar{0}$, since otherwise $\bar{x} > \bar{y}$ and $\bar{x}-\bar{y} \in \fJ$, $\Gamma(N-O_1)=(\bar{x}-\bar{y})+\Gamma(M_2)+\cdots+\Gamma(M_d) \in I_d$, contradicting Proposition \ref{P:weight estimation for modest/optimal elt} (iii). Let $0 \leq k \leq f-2$ be the least subscript such that $w_k>0$. We have the following result.

\begin{lemma} \label{L:main lemma 2}
  \begin{itemize}
    \item[(i)] $\inner{\bar{\psi}_i}{\bar{w}} < \inner{\bar{\psi}_i}{\bar{e}_{f-1}}$ for $0 \leq i \leq k$.
    \item[(ii)] $\floor{\eta_i}-\beta_i \geq d-2$ for all $i$ and there exists $k < l \leq {f-1}$ such that $\floor{\eta_l}-\beta_l=d-2$ and $\floor{\eta_i}-\beta_i \geq d-1$
    for $l-f < i \leq k$, i.e., $i=l+1, l+2, \ldots, f-1, 0, 1, \ldots, k$.
  \end{itemize}
\end{lemma}

\begin{proof}
  We show (i) by contradiction. For each $0 \leq i \leq k$, by definition of $\bar{\psi}_i$,
  \[
  \langle \bar{\psi}_i, \bar{e}_{f-1} \rangle=p^{f-1-i}
  \]
  and $\langle \bar{\psi}_i, \bar{w} \rangle$ is a sum of $p$-powers less than $p^{f-1-i}$ since $w_j=0$ for $-1 \leq j <i$. Suppose $\langle \bar{\psi}_i, \bar{w} \rangle \geq \langle \bar{\psi}_i, \bar{e}_{f-1} \rangle$ for some $0 \leq i \leq k$. Then there is a subset of $p$-powers in the sum $\langle \bar{\psi}_i, \bar{w} \rangle$ whose terms add up to $p^{f-1-i}$. In other words, there exists some $\bar{w}' \leq \bar{w}$ such that
  \begin{equation} \label{E:inner prod}
    \inner{\bar{\psi}_i}{\bar{w}'}=\inner{\bar{\psi}_i}{\bar{e}_{f-1}}.
  \end{equation}
  Another observation is that $\bar{w}$ represents those $p$-powers $p^b$ in $\cP(O_1) \setminus \cP(M_1)$. In particular, $b<a$ for each $b$ since $O_1<M_1$. $\bar{w}' \leq \bar{w}$ represents a subset of such $p$-powers. Let $M$ be the sum of $p$-powers represented by $\bar{w}'$, then $M<p^a$ and $\Gamma(M)=\bar{w}'$. Note that $\Gamma(p^a)=\bar{e}_{f-1}$, then \eqref{E:inner prod} implies
  \[
  M \equiv p^a \smod{q-1}.
  \]
  By the choice of $p^a$, we can find some $j>1$ such that $p^a \in \cP(O_j)$. Consider composition
  \[
  \bld{X}=(O_1-M+p^a, \ldots, O_j-p^a+M, \ldots, O_d),
  \]
  then $\bld{X} \in W_d(N)$ and $\wt(\bld{X})<\wt(\bld{O})$. Contradiction.

  To prove (ii), we note that $E(\bar{\eta}-\bar{\beta})=\Gamma(N-O_1)=\Gamma(O_2)+ \cdots + \Gamma(O_d) \in I_{d-1}$. Hence, by Proposition \ref{P:IJ} (i), $\eta_i-\beta_i > d-2$. Thus 
  \[
  \floor{\eta_i}-\beta_i \geq d-2
  \]
  for all $i$. On the other hand, $\Gamma(N-O_1) \not\in I_d$ implies $\min_i (\floor{\eta_i}-\beta_i) = d-2$. Let $l$ be the largest subscript such that $\floor{\eta_l}-\beta_l = d-2$. Then for $l < i < f$,
  \[
  \floor{\eta_i}-\beta_i \geq d-1.
  \]
  To finish the proof, we show $\floor{\eta_i}-\beta_i \geq d-1$ for $0 \leq i \leq k$. By construction, we have
  \[
  \bar{x}=\bar{v}+\bar{w}_1, \quad \bar{y}=\bar{v}+\bar{w},
  \]
  where $\bar{w}_1 \geq \bar{e}_{f-1}$. For $0 \leq i \leq k$,
  \begin{align*}
    \beta_i & = (q-1)^{-1}\inner{\bar{\psi}_i}{\bar{y}}\\
    & = (q-1)^{-1}(\inner{\bar{\psi}_i}{\bar{v}}+\inner{\bar{\psi}_i}{\bar{w}})\\
    (\text{by (i)}) & < (q-1)^{-1}(\inner{\bar{\psi}_i}{\bar{v}}+\inner{\bar{\psi}_i}{\bar{e}_{f-1}})\\
    & \leq (q-1)^{-1}(\inner{\bar{\psi}_i}{\bar{v}}+\inner{\bar{\psi}_i}{\bar{w}_1})\\
    & = (q-1)^{-1}\inner{\bar{\psi}_i}{\bar{x}} = \alpha_i.
  \end{align*}
  Note that $E(\bar{\eta}-\bar{\alpha})=\Gamma(N-M_1) \in I_{d-1}$, then by Proposition \ref{P:IJ} (i), we have $\floor{\eta_i}-\alpha_i > d-2$ for all $0 \leq i < f$. For $0 \leq i \leq k$, $\beta_i < \alpha_i$ by the above calculation, and thus $\floor{\eta_i}-\beta_i \geq d-1$.
\end{proof}

Define, for $1 \leq j \leq d$,
\[
\bar{u}_j=[u_{0, j}, \ldots, u_{f-1, j}]^t:=\sum_{s=j}^d\Gamma(O_s),\quad \bar{\theta}_j=[\theta_{0,j}, \ldots,\theta_{f-1, j}]^t:=E^{-1}\bar{u}_{j}.
\]
Note that $\bar{u}_1=\bar{u}$ and $\bar{\theta}_1=\bar{\eta}$. By Lemma \ref{L:main lemma 2} (ii), we have
\begin{equation} \label{E:theta_i,2}
  \floor{\theta_{l, 2}}=d-2, \quad \floor{\theta_{i, 2}} \geq d-1 \mbox{ for } l-f < i \leq k. 
\end{equation}
The following construction uses $\bar{\theta}_j$ to get a new composition $\bld{Z} \in W_d(N)$ whose weight is less than that of $\bld{O}$. Let $\bar{\phi}_1=[\phi_{i, 1}]^t:=\bar{\theta}_1$. Define $\bar{\phi}_2=[\phi_{i, 2}]^t$ inductively as following.
\[
\phi_{i, 2}=\begin{cases}\theta_{i, 2} & \text{for } k < i \leq l\\ \min(\theta_{i, 2}-1, p\phi_{i+1, 2}) & i=k, k-1, \ldots, 0, f-1, f-2, \ldots, l+1. \end{cases}
\]
For $j=3, \ldots, d$, define $\bar{\phi}_j=[\phi_{i, j}]^t$ recursively as
\[
\phi_{i, j}=\begin{cases}\theta_{i, j} & \text{for } k < i \leq l\\ \min(\phi_{i, j-1}-1, p\phi_{i+1, j}) & i=k, k-1, \ldots, 0, f-1, f-2, \ldots, l+1. \end{cases}
\]

\begin{prop} \label{P:mainprop}
  For $1 \leq j \leq d$, let
  \[
  \bar{z}_j=[z_{0, j}, \ldots, z_{f-1, j}]^t:= E\bar{\phi}_j.
  \]
  Then
  \begin{itemize}
    \item[(i)] $\bar{\phi}_{j}-\bar{\phi}_{j+1} \in \ZZ_+^f$ for $1 \leq j \leq d-1$.
    \item[(ii)] $\bar{z}_{j} \in \ZZ^f$ for all $j$.
    \item[(iii)] $\min_{0 \leq i \leq f-1}(\floor{\phi_{i, j}})=\floor{\phi_{l, j}}=d-j$ for $2 \leq j \leq d$.
    \item[(iv)] $z_{k, 2}=u_{k, 2}+1 \leq u_k$.
    \item[(v)] $0 \leq z_{l, 2} \leq u_{l, 2}-p$.
    \item[(vi)] $0 \leq z_{i, 2} \leq \max(u_{i, 2}-(p-1), 0)$ for $l-f < i < k$.
    \item[(vii)] $z_{i, j}=u_{i, j}$ for $k<i<l$ and $2 \leq j \leq d$.
    \item[(viii)] $0 \leq z_{i, j} \leq \max(z_{i, j-1}-(p-1), 0)$ for $l-f \leq i \leq k$ and $3 \leq j \leq d$.
  \end{itemize}
\end{prop}

\begin{proof}
  (i): By construction, $\phi_{i,j}-\phi_{i,j+1} > 0$ for all $i,j$. Hence it is enough to show 
  \begin{itemize}
      \item[(a)] $\bar{\theta}_j - \bar{\theta}_{j+1} \in \ZZ^f$ for $1 \leq j \leq d-1$;
      \item[(b)] $\{\phi_{i, j}\}=\{\theta_{i, j}\}$ for all $i, j$, where $\{x\}$ is the fractional part of $x$.
  \end{itemize}
  We note that for $1 \leq j \leq d-1$
  \[
  E(\bar{\theta}_j-\bar{\theta}_{j+1})=\Gamma(O_j) \in \fJ.
  \]
  This implies (a) by Proposition \ref{P:integral}. (a) says  $\{\theta_{i,j}\}=\{\theta_{i, j+1}\}$ for all $i$ and $j$. Also $p\{\theta_{i+1, j}\}-\{\theta_{i, j}\} \in \ZZ$ since $p\theta_{i+1, j}-\theta_{i, j}=u_{i, j} \in \NN$. With these two properties in mind, starting with the initial case $\{\phi_{i, 1}\}=\{\theta_{i, 1}\}$ since $\bar{\phi}_1=\bar{\theta}_1$, following the inductive construction of $\phi_{i, j}$, one can check that (b) holds.

  (ii): Since $\bar{\phi}_1=\bar{\eta}$, $\bar{z}_1=E\bar{\eta}=\bar{u} \in \ZZ^f$. For $j > 1$,
  \[
  \bar{z}_j=E\bar{\phi}_1-\sum_{s=1}^{j-1} E(\bar{\phi}_s-\bar{\phi}_{s+1}).
  \]
  By (i), $\bar{\phi}_s-\bar{\phi}_{s+1} \in \ZZ^f$ for each $s$, hence $\bar{z}_j \in \ZZ^f$.

  (iii): We first show that $\floor{\phi_{i,j}} \geq d-j$ for $2 \leq j \leq d$ and $k < i \leq l$. This is the same as showing  
  \[
  \floor{\theta_{i,j}} \geq d-j
  \]
  for $2 \leq j \leq d$, since $\phi_{l,j}=\theta_{i,j}$ by construction. Note that for each $j$, $E\bar{\theta}_j=\sum_{s=j}^d\Gamma(O_j) \in I_{d-j+1}$. Thus the statement follows from Proposition \ref{P:IJ}.
  
  Next, we prove $\floor{\phi_{l,j}} = d-j$ for each $j$. The $j=2$ case is given by \eqref{E:theta_i,2}. For $3 \leq j \leq d$, $\bar{\theta}_2-\bar{\theta}_j=\sum_{s=2}^{j-1} E^{-1}\Gamma(O_s)$. By Propositions \ref{P:estimation on M/O} and \ref{P:IJ}, $\theta_{l, 2}-\theta_{l, j} \geq j-2$, which implies $\floor{\theta_{l, j}} \leq \floor{\theta_{l, 2}}-(j-2)=d-j$. Thus the statement follows since $\floor{\theta_{l, j}} \geq d-j$.
  
  Last, we show $\floor{\phi_{i,j}} \geq d-j$ for $l-f < i \leq k$ by induction on $j$. For $j=2$, we have 
  \[
  \floor{\theta_{i, 2}-1} \geq d-2
  \]
  for $l-f < i \leq k$ by \eqref{E:theta_i,2}. Taking $i=k$, since $\floor{p\phi_{k+1, 2}} = \floor{p\theta_{k+1, 2}} \geq d-2$, we get
  \[
  \floor{\phi_{k, 2}}=\min(\floor{\theta_{k, 2}-1}, \floor{p\phi_{k+1, 2}}) \geq d-2.
  \]
  Note that 
  \begin{align*}
    \floor{\phi_{i+1, 2}} \geq d-2 & \Rightarrow \floor{p\phi_{i+1, 2}} \geq d-2\\
    & \Rightarrow \floor{\phi_{i, 2}}=\min(\floor{\theta_{i, 2}-1}, \floor{p\phi_{i+1, 2}}) \geq d-2.
  \end{align*}
  Hence, a backwards induction on $i$ starting from $k$ implies that for $k \geq i > l-f$,
  \[
  \floor{\phi_{i,2}} \geq d-2.
  \]
  Suppose $\floor{\phi_{i, j-1}} \geq d-j+1$ for $l-f < i \leq k$. Then $\floor{\phi_{k,j-1}-1} \geq d-j$ and $\floor{p\phi_{k+1, j}} \geq d-j$ since $\floor{\phi_{k+1, j}} \geq d-j$ by previous statement. This implies $\floor{\phi_{k, j}}=\min(\floor{\phi_{k, j-1}-1}, \floor{p\phi_{k+1, j}}) \geq d-j$. Similarly, we have
  \begin{align*}
    \floor{\phi_{i+1, j}} \geq d-j & \Rightarrow \floor{p\phi_{i+1, j}} \geq d-j \\
    & \Rightarrow \floor{\phi_{i, j}}=\min(\floor{\phi_{i, j-1}-1}, \floor{p\phi_{i+1, j}}) \geq d-j.
  \end{align*}
  Again, a backwards induction on $i$ shows that $\floor{\phi_{i, j}} \geq d-j$ for $k \geq i  > l-f$.

  (iv): Since
  \[
  p\phi_{k+1, 2}-(\theta_{k, 2}-1)=p\theta_{k+1, 2}-\theta_{k, 2}+1=u_{k, 2}+1 > 0,
  \]
  $\phi_{k, 2}=\theta_{k, 2}-1$ and $z_{k, 2}=u_{k, 2}+1$. By construction,
  \[
  u_{k, 2}=u_k-y_k \leq u_k-w_k \leq u_k-1,
  \]
  so $u_{k, 2}+1\leq u_k$.

  (v): The second inequality is given by
  \[
  z_{l, 2}=p\phi_{l+1, 2}-\phi_{l, 2} \leq p(\theta_{l+1, 2}-1)-\theta_{l, 2}=u_{l, 2}-p.
  \]
  To show $z_{l, 2} \geq 0$, we have $\min_{0 \leq i \leq f-1}(\floor{\phi_{i, 2}})=\floor{\phi_{l, 2}}=d-2$ by (iii). This implies
  \[
  z_{l, 2}=p\phi_{l+1, 2}-\phi_{l, 2} \geq p(d-2)-(d-2)-\{\phi_{l, 2}\} \geq 0.
  \]
  
  (vi): For $l-f < i < k$, if $\phi_{i, 2}=p\phi_{i+1, 2}$, $z_{i, 2}=p\phi_{i+1, 2}-\phi_{i, 2}=0$; otherwise, $\phi_{i, 2}=\theta_{i, 2}-1$ and $z_{i, 2} = p\phi_{i+1, 2}-(\theta_{i, 2}-1) \leq p(\theta_{i+1, 2}-1)-(\theta_{i, 2}-1)=u_{i, 2}-(p-1).$

  (vii): This follows directly from the construction of the $\bar{\phi}_j$'s.
  
  (viii): We break up the proof into three cases.
  
  For $l-f < i < k$, we only need to check for the case where $\phi_{i, j}=\phi_{i, j-1}-1$, since otherwise $z_{i, j}=0$. In this case, $\phi_{i, j}-1 \leq p\phi_{i+1, j}$ and
  \[
  0 \leq z_{i, j}=p\phi_{i+1, j}-(\phi_{i, j-1}-1) \leq p(\phi_{i+1, j-1}-1)-(\phi_{i, j-1}-1)=z_{i, j-1}-(p-1).
  \]
  
  For $i=k$, again, we may assume $\phi_{k, j}=\phi_{k, j-1}-1$, then
  \begin{align*}
    0 \leq z_{k, j} & = p\theta_{k+1, j}-(\phi_{k, j-1}-1)\\
    & \leq p(\theta_{k+1, j-1}-1)-(\phi_{k, j-1}-1)=z_{k, j-1}-(p-1),
  \end{align*}
  where the second inequality follows from the fact that $\bar{\theta}_{j-1}-\bar{\theta}_j=E^{-1}\Gamma(O_{j-1}) \geq 1$ by Propositions \ref{P:estimation on M/O} and \ref{P:IJ}. 
  
  For $i=l$, by (iii), we have
  \[
  z_{l, j}=p\phi_{l+1, j}-\phi_{l, j} \geq p(d-j)-(d-j)-\{\phi_{l, j}\} \geq 0.
  \]
  Finally, $z_{l, j}=p\phi_{l+1, j}-\theta_{l, j} \leq p(\phi_{l+1,j-1}-1)-(\theta_{l, j-1}-1)=z_{l, j-1}-(p-1).$
\end{proof}

Proposition \ref{P:mainprop} implies that the matrix
\[
B=[\bar{z}_1-\bar{z}_2, \ldots, \bar{z}_{d-1}-\bar{z}_d, \bar{z}_d]
\]
is a valid matrix of $W_d(N)$. Let $\bld{Z}=(Z_1, \ldots, Z_d)$ be the $\tau$-monotonic element in $W_d^B(N)$. We show that $\wt(\bld{Z}) < \wt(\bld{O})$ and hence get a contradiction. 

\subsubsection{Estimation on $\wt(\bld{Z})$.} \label{subsub:weight estimation}

For $2 \leq j \leq d$, define
\begin{align*}
  Z'_j & := Z_j+Z_{j+1}+\cdots+Z_d,\\
  O'_j & := O_j+O_{j+1}+\cdots+O_d.
\end{align*}
Then $\Gamma(Z'_j)=\bar{z}_j$ and $\Gamma(O'_j)=\bar{u}_j$. And weights of $\bld{Z}$ and $\bld{O}$ can be expressed as
\begin{align*}
  \wt(\bld{Z}) & = N+Z'_2+\cdots+Z'_d,\\
  \wt(\bld{O}) & = N+O'_2+\cdots+O'_d.
\end{align*}
To describe these $Z'_j, O'_j$ explicitly, for each $0 \leq i \leq f-1$, denote
\[
\tau_i(N)=(\tau_{i, u_i}, \tau_{i, u_i-1}, \ldots, \tau_{i, 1}).
\]
We recall that $\tau_i(N)$ is defined as the subsequence of the nonincreasing sequence of $p$-powers in $\cP(n)$, where the exponents of powers in it are congruent to $i$ modulo $f$. 

Let $\tau_{i, 0}=0$. Then, by $\tau$-monotonicity, we have
\[
Z'_j=\sum_{i=0}^{f-1}\sum_{s=0}^{z_{i, j}}\tau_{i, s}, \quad O'_j=\sum_{i=0}^{f-1}\sum_{s=0}^{u_{i, j}}\tau_{i, s}.
\]
By Proposition \ref{P:mainprop} (vii),
\[
O'_j-Z'_j=\sum_{i \in I} \left( \sum_{s=0}^{u_{i, j}}\tau_{i, s}-\sum_{s=0}^{z_{i, j}}\tau_{i, s}\right),
\]
where
\[
I=\{ l, \ldots, f-1, 0, \ldots, k \}.
\]
For $j=2$ and $i \in I$, we have the following:
\begin{itemize}
    \item[(1)] By Proposition \ref{P:mainprop} (iv, v, vi), $z_{k, 2}=u_{k, 2}+1$ and for $i \in I\backslash\{k\}$, $z_{i, 2} \leq  u_{i, 2}$ where ``=" holds iff $z_{i, 2}=u_{i, 2}=0$.
    \item[(2)] $\tau_{f-1, u_{f-1, 2}}=p^a$ since it is the largest $p$-power not in $\cP(O_1)$ whose exponent is $f-1$ mod $n$. In particular, $u_{f-1, 2}>0$ and hence $z_{f-1, 2}<u_{f-1, 2}$ by (1).
    \item[(3)] Let $\tau_{k, z_{k, 2}}=p^b$, then $z_{k, 2}=u_{k, 2}+1$ implies that $p^b$ is the last $p$-power in $\tau_k(O_1)$. By our choice of $k$, $p^b \in \cP(O_1) \setminus \cP(M_1)$. In particular, $p^b<p^a$.
\end{itemize}
With these observations, we have
\begin{equation*}  
  O'_2-Z'_2 \geq \tau_{f-1, u_{f-1, 2}}-\tau_{k, z_{k, 2}}+\sum_{i \in I\backslash\{k, f-1\}}\tau_{i, z_{i, 2}}=p^a-p^b+\sum_{i \in I\backslash\{k, f-1\}}\tau_{i, z_{i, 2}}.
\end{equation*} 
Thus
\begin{equation} \label{E:wt diff}     \wt(\bld{O})-\wt(\bld{Z})=\sum_{j=2}^d O'_j-Z'_j\geq p^a-p^b+\sum_{i \in I\backslash\{k, f-1\}}\tau_{i, z_{i, 2}}+\sum_{j=3}^d O'_j-Z'_j. 
\end{equation}
The next lemma gives a lower bound for $\sum_{j=3}^d O'_j-Z'_j$.

\begin{lemma} \label{L:diff bound}
  Let $I=\{ l, \ldots, f-1, 0, \ldots, k \}$, then
  \[
  \sum_{j=3}^d O'_j-Z'_j>-\sum_{i \in I} \tau_{i, z_{i, 2}}.
  \]
\end{lemma}

\begin{proof}
  We note that $\tau_{k, z_{k, 2}}>0$ since $z_{k,2} > 0$ by Proposition \ref{P:mainprop} (iv). The statement is trivial if $z_{i,j}=0$ for all $i \in I$ and $3 \leq j \leq d$. Assuming they are not all vanishing, the statement follows from the following calculation.
  \begin{align*}
    \sum_{j=3}^d O'_j-Z'_j &= \sum_{j=3}^d \sum_{i \in I}\left(\sum_{s=0}^{u_{i,j}} \tau_{i, s}-\sum_{s=0}^{z_{i, j}}\tau_{i, s} \right)\\
    & > -\sum_{i \in I}\sum_{j=3}^d\sum_{s=0}^{z_{i, j}}\tau_{i, s}\\
    & > -p\sum_{i \in I}\sum_{j=3}^d\tau_{i, z_{i,j}}\\
    & > -p^2/q\sum_{i \in I}\tau_{i, z_{i,2}}\\
    & \geq -\sum_{i \in I}\tau_{i, z_{i,2}}.
  \end{align*}
  The first inequality is trivial. The last one follows from the assumption $f \geq 2$. For the second inequality, we note that each $p$-power appearing in the sum $\sum_{s=0}^{z_{i, j}}\tau_{i, s}$ repeats at most $p-1$ times and the largest term is $\tau_{i, z_{i,j}}$, which indicates $\sum_{s=0}^{z_{i, j}}\tau_{i, s} < p\tau_{i, z_{i,j}}$. By a similar argument and Proposition \ref{P:mainprop} (viii), for all $i \in I $ and $3 \leq j \leq d$, we have $\tau_{i, z_{i,j}} \leq q^{-1}\tau_{i, z_{i, j-1}}$, thus
  \[
  \sum_{j=3}^d\tau_{i, z_{i,j}} \leq q^{-1}\sum_{j=2}^{d-1}\tau_{i, z_{i, j}} < \frac{p}{q} \tau_{i, z_{i,2}}.
  \]
  This gives the third inequality.
\end{proof}

Now we are ready to claim the contradiction, which finishes the proof of Theorem \ref{T:modest'}.

\begin{prop}
  $\wt(\bld{Z})<\wt(\bld{O})$.
\end{prop}

\begin{proof}
  We first show $\tau_{f-1, z_{f-1,2}} < p^a$. Assume $z_{f-1,2}>0$ since otherwise $\tau_{f-1, z_{f-1,2}}=0$. Then $\tau_{f-1, z_{f-1,2}} \leq q^{-1}\tau_{f-1, u_{f-1,2}}$ because $z_{f-1, 2} \leq u_{f-1, 2}-(p-1)$ by Proposition \ref{P:mainprop}. Note that $\tau_{f-1, u_{f-1,2}}=p^a$ as we mentioned earlier, thus $\tau_{f-1, z_{f-1,2}} < p^a$. 

  Putting together \eqref{E:wt diff} and Lemma \ref{L:diff bound}, we have
  \begin{equation*} 
    \wt(\bld{O})-\wt(\bld{Z}) > p^a-2p^b-\tau_{f-1, z_{f-1,2}}.
  \end{equation*}
  Here $p^b$ and $\tau_{f-1, z_{f-1,2}}$ are $p$-powers less than $p^a$ and they are distinct since their exponents fall into different residue classes mod $f$. We break up the proof into 3 cases.
  
  (1) $p \geq 3$: $\wt(\bld{O})-\wt(\bld{Z})>0$ since $2p^b+\tau_{f-1, z_{f-1,2}}<p^a$. 
  
  (2) $p=2$ and $b+1<a$: $2p^b+\tau_{f-1, z_{f-1,2}}=p^{b+1}+\tau_{f-1, z_{f-1,2}} \leq p^a$, thus $\wt(\bld{O})-\wt(\bld{Z})>0$. 
  
  (3) $p=2$ and $b+1=a$: In this case, we have $k+1=l=f-1$ and $I=\{0, \ldots, f-1 \}$.
  By Proposition \ref{P:mainprop} (v), $u_{f-1, 2}-z_{f-1, 2} \geq 2$. Hence
  \begin{align*}
    O'_2-Z'_2 & = \left(\sum_{s=0}^{u_{f-1,2}} \tau_{i, s} -\sum_{s=0}^{z_{f-1, 2}} \tau_{i, s} \right)+\sum_{i=0}^{k} \left(\sum_{s=0}^{u_{i,2}} \tau_{i, s} -\sum_{s=0}^{z_{i, 2}} \tau_{i, s} \right)\\
    & \geq  p^a + \tau_{f-1, 1+z_{f-1, 2}}- p^b +\sum_{i=0}^{k-1} \tau_{i, z_{i, 2}}
  \end{align*} 
  and by Lemma \ref{L:diff bound},
  \begin{align*}
    \wt(\bld{O})-\wt(\bld{Z}) & = \sum_{j=2}^d O'_j-Z'_j \\
    & >  p^a+\tau_{f-1, 1+z_{f-1, 2}}-p^b+\sum_{i=0}^{k-1} \tau_{i, z_{i, 2}}-\sum_{i \in I} \tau_{i, z_{i, 2}}\\
    & = p^a+\tau_{f-1, 1+z_{f-1, 2}}-p^b-p^b-\tau_{f-1, z_{f-1, 2}}\\
    & = \tau_{f-1, 1+z_{f-1, 2}}-\tau_{f-1, z_{f-1, 2}} \geq 0. \qedhere
  \end{align*}
\end{proof}

\nocite{Book:Goss-1996, Thakur-2017}
\bibliographystyle{amsalpha}
\bibliography{ref}
\end{document}